\documentclass[12pt]{amsart}

\usepackage{amssymb,latexsym,bm}
\usepackage{graphicx,psfrag}
\usepackage{enumitem}
\setlist[enumerate]{parsep=0pt plus 4pt,topsep=0pt plus 4pt}
\usepackage{url}
\usepackage{hyperref}
\allowdisplaybreaks

\newcommand{\excise}[1]{}

\newtheorem{thm}{Theorem}[section]
\newtheorem{lemma}[thm]{Lemma}

\newtheorem{cor}[thm]{Corollary}
\newtheorem{prop}[thm]{Proposition}

\theoremstyle{definition}
\newtheorem{example}[thm]{Example}
\newtheorem{remark}[thm]{Remark}
\newtheorem{defn}[thm]{Definition}

\numberwithin{equation}{section}

\newcommand{\Ring}[1]{\ensuremath{\mathbb{#1}}}

\renewcommand\>{\rangle}
\newcommand\<{\langle}
\newcommand\0{\mathbf{0}}
\newcommand\FF{\Ring{F}}
\newcommand\KK{\Ring{K}}
\newcommand\NN{\Ring{N}}
\newcommand\OO{\mathcal{O}}

\newcommand\RR{\Ring{R}}

\newcommand\ee{{\bm\ve}}
\newcommand\kk{\Bbbk}
\newcommand\mm{{\mathfrak m}}

\newcommand\xx{{\mathbf x}}
\newcommand\yy{{\mathbf y}}

\newcommand\oG{\mathring G}
\newcommand\oK{\mathring\KK}
\newcommand\ve{\varepsilon}

\renewcommand\aa{{\mathbf a}}
\renewcommand\phi{\varphi}

\newcommand\from{\leftarrow}
\newcommand\orho{\mathring\rho}
\newcommand\into{\hookrightarrow}
\newcommand\otni{\hookleftarrow}
\newcommand\onto{\twoheadrightarrow}

\newcommand\spot{{\hbox{\raisebox{1pt}{\tiny$\scriptscriptstyle\bullet$}}}}

\newcommand\ffrom{\longleftarrow}
\newcommand\minus{\smallsetminus}

\newcommand\dirlim{\varinjlim}
\newcommand\invlim{\varprojlim}

\renewcommand\epsilon{\varepsilon}

\newcommand\mkl[1]{\makebox[0pt][l]{$#1$}}

\newcommand{\aoverb}[2]{{\genfrac{}{}{0pt}{1}{#1}{#2}}}
\def\twoline#1#2{\aoverb{\scriptstyle {#1}}{\scriptstyle {#2}}}

\DeclareMathOperator\Ext{Ext} 
\DeclareMathOperator\Hom{Hom} 
\DeclareMathOperator\Tor{Tor} 
\DeclareMathOperator\Tot{Tot} 
\DeclareMathOperator\fldim{fl.\!dim} 
\DeclareMathOperator\gldim{gl.\!dim} 

\newcommand\dis{\displaystyle}

\begin{document}

\mbox{}
\vspace{-4.5ex}
\title{\!Global dimension of real-exponent polynomial rings\!}
\author{Nathan Geist}
\address{Mathematics Department\\Duke University\\Durham, NC 27708}
\email{nathan.geist@duke.edu}
\author{Ezra Miller}
\address{Mathematics Department\\Duke University\\Durham, NC 27708}
\urladdr{\url{https://math.duke.edu/people/ezra-miller}}

\makeatletter
  \@namedef{subjclassname@2020}{\textup{2020} Mathematics Subject Classification}
\makeatother
%
%
%

\date{15 May 2022}

\begin{abstract}
The ring $R$ of real-exponent polynomials in $n$ variables over any
field has global dimension~$n + 1$ and flat dimension~$n$.  In
particular, the residue field $\kk = R/\mm$ of~$R$ modulo its maximal
graded ideal~$\mm$ has flat dimension~$n$ via a Koszul-like
resolution.  Projective and flat resolutions of all $R$-modules are
constructed from this resolution of~$\kk$.  The same results hold when
$R$ is replaced by the monoid algebra for the positive cone of any
subgroup of~$\RR^n$ satisfying a mild density condition.\vspace{-4ex}
\end{abstract}
\maketitle


\section{Introduction}\label{s:intro}

\subsection*{Overview}
The aim of this note is to prove that the commutative ring $R$ of
real-exponent polynomials in $n$ variables over any field~$\kk$ has
global dimension~$n + 1$ and flat dimension~$n$
(Theorem~\ref{t:gl.dim} and Corollary~\ref{c:fldim}).  It might be
unexpected that $R$ has finite global dimension at all, but it should
be more expected that the flat dimension is achieved by the residue
field $\kk = R/\mm$ of~$R$ modulo its maximal graded ideal~$\mm$; a
Koszul-like construction shows that it is
(Proposition~\ref{p:fldim(k)} along with
Example~\ref{e:usual-koszul}).  In one real-exponent variable the
residue field~$\kk$ also achieves the global dimension bound of~$2$
(Lemma~\ref{l:Ext(k,F)}), and this calculation lifts to $n$ variables
by tensoring with an ordinary Koszul complex
(Proposition~\ref{p:n+1}), demonstrating global dimension at least~$n
+ 1$.  Projective and flat resolutions of all $R$-modules are
constructed from resolutions of the residue field in the proofs of
Theorems~\ref{t:gl.dim} and~\ref{t:flat-res} to yield the respective
upper bounds of $n + 1$ and~$n$.  The results extend to the monoid
algebra for the positive cone of any subgroup of~$\RR^n$ satisfying a
mild density condition (Definition~\ref{d:G}
and~Theorem~\ref{t:G-flat-res}).

\subsection*{Background}
Global dimension measures how long projective resolutions of modules
can get, or how high the homological degree of a nonvanishing Ext
module can be \cite[Theorem~4.1.2]{weibel1994}.  Finding rings of
finite global dimension is of particular value, since they are
considered to be smooth, generalizing the best-known case of local
noetherian commutative rings \cite{auslander-buchsbaum1957,serre1956},
which correspond to germs of functions on nonsingular algebraic
varieties.

The related notion of flat dimension (also called Tor dimension or
weak global dimension) measures how long flat resolutions of modules
can get, or how high the homological degree of a nonvanishing Tor
module can be.  Flat dimension is bounded by global dimension because
projective modules are flat.  These two dimensions agree for
noetherian commutative rings \cite[Proposition~4.1.5]{weibel1994}.
Without the noetherian condition equality can fail; commutative
examples include von Neumann regular rings that are infinite products
of fields (see~\mbox{\cite[p.\,98]{weibel1994}}), but domains are
harder~to~come~by.

The cardinality of a real-exponent polynomial ring a~priori indicates
a difference between flat and projective dimension that could be as
high as $1$ plus the index on~$\aleph$ in the cardinality of the real
numbers \cite[p.14]{osofsky1974}.  In certain situations, such as in
valuation rings, ideals generated by $\aleph_n$ and no fewer elements
are known to cause global dimension at least $n+2$ \cite{osofsky1967}
(cf.~\cite[Theorem, p.14]{osofsky1974}).  But despite $R$ having an
ideal minimally generated by all monomials with total degree~$1$, of
which there are~$2^{\aleph_0}$,
the dimension of the positive cone of exponents is more pertinent than
its cardinality.  This remains the case when the exponent set is
intersected with a suitably dense subgroup of~$\RR^n$: the rank of the
subgroup is irrelevant (Section~\ref{s:dense}).

\subsection*{Methods}
The increase from global dimension~$n$ to $n + 1$ in the presence of
$n$~variables is powered by the violation of condition~5 from
\cite[Theorem~P]{bass1960}: a monomial ideal with an ``open orthant''
of exponents, such as the maximal ideal~$\mm_1$ in one indeterminate,
is a direct limit of principal monomial ideals
(Lemma~\ref{l:orthant-res}) but is not projective
(Lemma~\ref{l:Ext(k,F)}).  This phenomenon occurs already for Laurent
polynomials~$L_1$ in one integer-exponent variable.  But although
$\mm_1$ and~$L_1$ both have projective dimension~$1$, the
real-exponent maximal ideal~$\mm_1$ is a submodule of a projective
(actually, free) module; the inclusion has a cokernel, and its
projective dimension is greater by~$1$.

The most nontrivial point is how to produce a projective resolution of
length at most~$n + 1$ for any module over the real-exponent
polynomial ring~$R$ in $n$ variables.  Our approach takes two steps.
The first is a length~$n$ Koszul-like complex (Definition~\ref{d:y})
in $2n$ variables that resolves the residue field and can be massaged
into a flat resolution of any module (Theorem~\ref{t:flat-res}).  This
``total Koszul'' construction was applied to combinatorially resolve
monomial ideals in ordinary (that is, integer-exponent) polynomial
rings \cite[Section~6]{sylvan-resolution}.  The integer grading in the
noetherian case makes this construction produce a Koszul double
complex, which is key for the combinatorial purpose of minimalizing
the resulting free resolution by splitting an associated spectral
sequence.  It is not obvious whether the double complex survives to
the real-exponent setting, but the total complex does
(Definition~\ref{d:y}; cf.~\cite[Application~4.5.6]{weibel1994}), and
that suffices here because minimality is much more subtle---if it is
even possible---in the presence of real exponents
\cite{essential-real}.

\subsection*{Motivations}
Beyond basic algebra, there has been increased focus on non-noetherian
settings in, for example, noncommutative geometry and topological data
analysis.

Quantum noncommutative toric geometry
\cite{katzarkov-lupercio-meersseman-verjovsky2020} is based on dense
finitely generated additive subgroups of~$\RR^n$ instead of the
discrete sublattices that the noetherian commutative setting requires.
The situations treated by our main theorems, including especially
Section~\ref{s:dense}, correspond to ``smooth'' affine quantum toric
varieties and could have consequences for sheaf theory in that
setting.

The question of finite global dimension over real-exponent polynomial
rings has surfaced in topological data analysis (TDA), where modules
graded by $\RR^n$ are known as real multiparameter persistent
homology, cf.~\cite{lesnick-interleav2015}, \cite{essential-real}, and
\cite{bubenik-milicevic2020}, for example, or \cite{oudot2015} for a
perspective from quiver theory.  The question of global dimension
arises because defining metrics for statistical analysis requires
distances between persistence modules, many of which use derived
categorical constructions
\cite{kashiwara-schapira2018,strat-conical,berkouk-petit2021}; see
\cite[Section~7.1]{bubenik-milicevic2020} for an explicit mention of
the finite global dimension problem.

Real-exponent modules that are graded by $\RR^n$ and satisfy a
suitable finiteness condition (``tameness'') to replace the too-easily
violated noetherian or finitely presented conditions admit finite
multigraded resolutions by monomial ideals
\cite[Theorem~6.12]{hom-alg-poset-mods}, which are useful for~TDA.
But even in the tame setting no universal bound is known for the
finite lengths of such resolutions
\cite[Remark~13.15]{essential-real}.  The global dimension
calculations here suggest but do not immediately imply a universal
bound~of~$n +\nolinebreak 1$.

\subsection*{Notation}
The ordered additive group $\RR$ of real numbers has its monoid
$\RR_+$ of nonnegative elements.  The $n$-fold product $\RR^n =
\prod_{i=1}^n \RR$ has nonnegative cone $\RR^n_+ = \prod_{i=1}^n
\RR_+$.  The monoid algebra $R = R_n = \kk[\RR^n_+]$ over an arbitrary
field~$\kk$ is the ring of real-exponent polynomials in $n$ variables:
finite sums $\sum_{\aa \in \RR^n_+} c_\aa \xx^\aa$, where $\xx^\aa =
x_1^{a_1} \cdots x_n^{a_n}$.  Its unique monoid-graded maximal
ideal~$\mm$ is spanned over~$\kk$ by all nonunit monomials.

Unadorned tensor products are over~$\kk$.  For example, $R \cong R_1
\otimes \cdots \otimes R_1$ is an $n$-fold tensor product over~$\kk$,
where $R_1 = \kk[\RR_+]$ is the real-exponent polynomial ring in one
variable with graded maximal ideal~$\mm_1$.

\section{Flat dimension~\texorpdfstring{$n$}{n}}\label{s:fldim}

\begin{lemma}\label{l:k-res-R1}
The filtered colimit $\dirlim_{\ve > 0} (R_1 \otni \<x^\ve\>)$ of the
inclusions of the principal ideals generated by $x^\ve$ for
positive~$\ve \in \RR$ is a flat resolution $\oK^1_\spot: R_1 \otni
\mm_1$ of\/~$\kk$ over~$R_1$.
\end{lemma}
\begin{proof}
Colimits commute with homology so the colimit is a resolution.
Filtered colimits of free modules are flat by Lazard's criterion
\cite{lazard1964}, so the resolution is flat.
\end{proof}

\begin{defn}\label{d:open-koszul}
The \emph{open Koszul complex} is the tensor product $\oK^\xx_\spot =
\bigotimes_{i=1}^n \oK^1_\spot$ over the field~$\kk$ of $n$ copies of
the flat resolution in Lemma~\ref{l:k-res-R1}.  The $2^n$ summands of
$\oK^\xx_\spot$, each a tensor product of~$j$ copies of~$R_1$ and
$n-j$ copies of~$\mm_1$, are \emph{orthant ideals}.
\end{defn}

\begin{example}\label{e:open-koszul}
The open Koszul complex in two real-exponent variables is depicted in
Figure~\ref{f:open-koszul}.  {}From a geometric perspective, take the
ordinary Koszul complex from Figure~\ref{f:ordinary-koszul}, replace
the free modules with their continous versions, and push the
generators as close to the origin as possible without meeting it.  The
four possible orthant ideals are rendered in
Figure~\ref{f:open-koszul}.  {}From left to right,
viewing them as tensor products, they correspond to the product of two
closed rays~$\kk[\RR_+]$, the product (in both orders) of a closed ray
with an open ray~$\mm$, and the product of two open rays.  In $n$~
real-exponent variables the $2^n$ orthant ideals arise from all
$n$-fold tensor products~of~closed~and~open~rays.
\begin{figure}[ht]
\centering
{%
\psfrag{0}{$0$}
\psfrag{from}{$\from$}
\psfrag{oplus}{$\oplus$}
\includegraphics[width=5.4in]{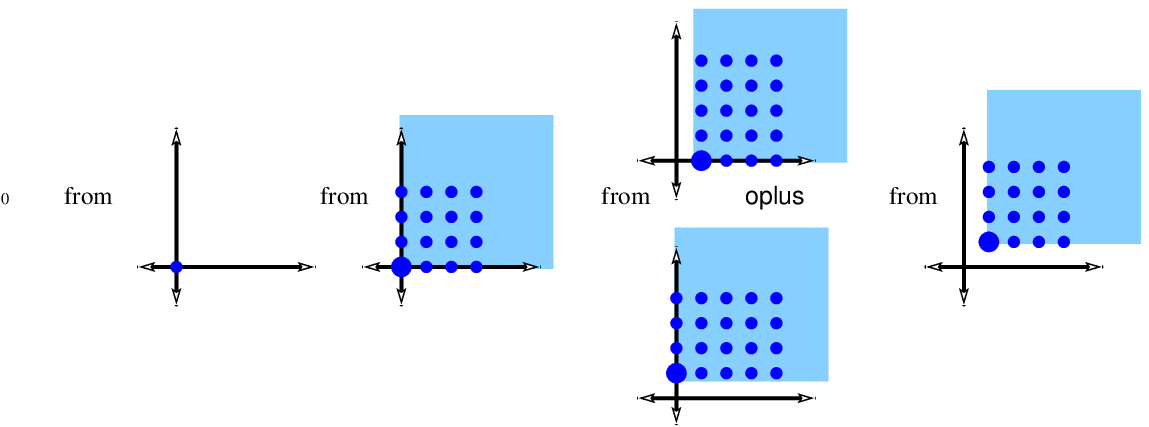}
}%
\vspace{-2ex}
\caption{Ordinary Koszul complex in two variables}
\label{f:ordinary-koszul}
\end{figure}

\begin{figure}[ht]
\centering
{%
\psfrag{0}{$0$}
\psfrag{from}{$\from$}
\psfrag{oplus}{$\oplus$}
\includegraphics[width=5in]{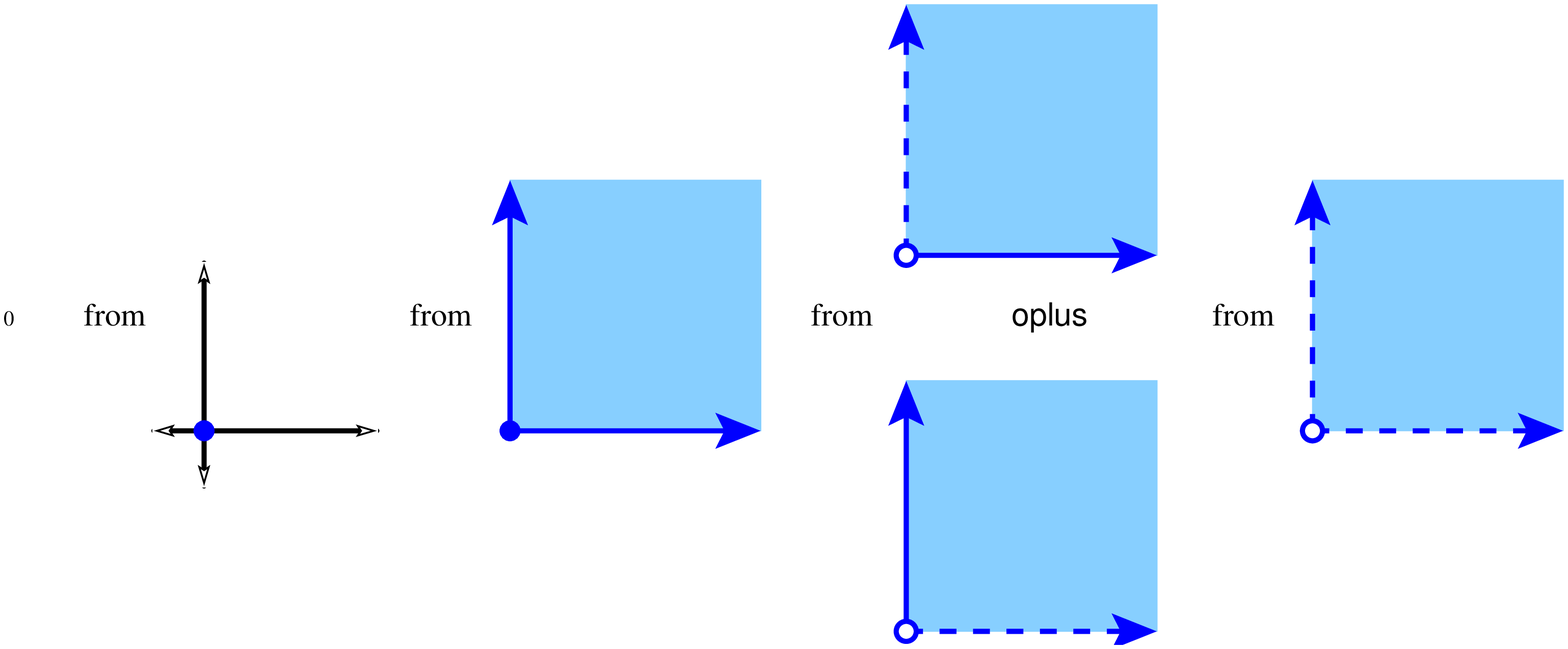}
}%
\vspace{-2ex}
\caption{Open Koszul complex in two real-exponent variables}
\label{f:open-koszul}
\end{figure}
\end{example}

\begin{prop}\label{p:fldim(k)}
The open Koszul complex~$\oK^\xx_\spot$ is a flat resolution
of\/~$\kk$ over~$R$.
\end{prop}
\begin{proof}
Lemma~\ref{l:k-res-R1} and the K\"unneth theorem
\cite[Theorem~3.6.3]{weibel1994}.
\end{proof}

Limit-Koszul complexes similar to~$\oK^\xx_\spot$ have previously been
used to compute flat dimensions of absolute integral closures
\cite{aberbach-hochster1997} in the context of tight closure.

\begin{example}\label{e:usual-koszul}
The sequence $\xx^{[\ve]} = x_1^\ve,\dots,x_n^\ve$ is regular in~$R$
\cite[Chapter~1]{bruns-herzog1998}, so the usual Koszul complex
$\KK_\spot(\xx^{[\ve]})$ is a length~$n$ free resolution of $B_n^\ve =
R/\<\xx^{[\ve]}\>$ over~$R$.  Using this resolution, $\Tor_n^R(\kk,
B_n^\ve) = \kk$ because $\kk \otimes_R \KK_\spot(\xx^{[\ve]})$ has
vanishing~differentials.
\end{example}

\begin{lemma}\label{l:Rotimes2}
The real-exponent polynomial ring $R^{\otimes 2} = R \otimes R$ has
$2n$~variables
\begin{align*}
& \xx = x_1,\dots,x_n = x_1 \otimes 1, \dots, x_n \otimes 1
\\
\text{and}\quad
& \yy = \hspace{.15ex}y_1,\dots,\hspace{.15ex}y_n
      = 1 \otimes \hspace{.2ex}x_1, \dots, 1 \otimes \hspace{.2ex}x_n.
\end{align*}
Over $R^{\otimes 2}$ is a directed system of Koszul complexes
$\KK_\spot(\xx^{[\ve]}-\yy^{[\ve]})$ on the sequences
$$%
  \xx^{[\ve]} - \yy^{[\ve]} = x_1^\ve - y_1^\ve, \dots, x_n^\ve - y_n^\ve
$$
with $\ve > 0$.  The colimit
$\dis%
  \oK^{\xx-\yy}_\spot = \dirlim_{\ve > 0} \KK_\spot(\xx^{[\ve]}-\yy^{[\ve]})
$
is an $R^{\otimes 2}$-flat resolution of~$R$.
\end{lemma}
\begin{proof}
The general case is the tensor product over~$\kk$ of $n$ copies of the
case $n = 1$, which in turn reduces to the calculation $R^{\otimes
2}/\<x^\ve - y^\ve \mid \ve > 0\> \cong R$.
\end{proof}

\begin{defn}\label{d:y}
Denote by $R^\xx$ and $R^\yy$ the copies of~$R$ embedded in
$R^{\otimes 2}$ as $R \otimes 1$ and $1 \otimes R$.  Fix an
$R^\xx$-module~$M$.
\begin{enumerate}
\item%
Write $M^\yy$ for the corresponding $R^\yy$-module, with the $\xx$
variables renamed to~$\yy$.
\item%
The \emph{open total Koszul complex} of an $R^\xx$-module~$M$ is
$\dis%
  \oK^{\xx-\yy}_\spot(M)
  =
  \oK^{\xx-\yy}_\spot \otimes_{R^\yy} M^\yy.
$
\end{enumerate}
\end{defn}

\begin{remark}\label{r:orthant}
By Definition~\ref{d:open-koszul}, each of the $4^n$ summands
of~$\oK^{\xx-\yy}_\spot$ in Lemma~\ref{l:Rotimes2} is the tensor
product over~$\kk$ of an orthant $R^\xx$-ideal and an orthant
$R^\yy$-ideal.
\end{remark}

\begin{thm}\label{t:flat-res}
The open total Koszul complex $\oK^{\xx-\yy}_\spot(M)$ is a length~$n$
resolution of~$M$ over $R^{\otimes 2}$ for any $R^\xx$-module~$M$.
This resolution is flat over~$R^\xx$; more precisely, as an
$R^\xx$-module $\oK^{\xx-\yy}_\spot(M)$ is a direct sum of orthant
$R^\xx$-ideals.
\end{thm}
\begin{proof}
The tensor product $\oK^{\xx-\yy}_\spot \otimes_{R^\yy} M^\yy$ is over
$R^\yy$ and hence converts the orthant $R^\xx$-ideal decomposition for
$\oK^{\xx-\yy}$ afforded by Remark~\ref{r:orthant} into one for
$\oK^{\xx-\yy}_\spot(M)$.

Since tensor products commute with colimits, $\oK^{\xx-\yy}_\spot(M) =
\dirlim_{\ve > 0} \KK^\ve_\spot(M)$, where $\KK^\ve_\spot(M) =
\KK_\spot(\xx^{[\ve]} - \yy^{[\ve]}) \otimes_{R^\yy} M^\yy$.
Each complex $\KK^\ve_\spot(M)$ is the ordinary Koszul complex of the
sequence $\xx^{[\ve]} - \yy^{[\ve]}$ on the $R^{\otimes 2}$-module
$R^{\otimes 2} \otimes_{R^\yy} M^\yy$.  But $\xx^{[\ve]} -
\yy^{[\ve]}$ is a regular sequence on this module because the $\xx$
variables are algebraically independent from the $\yy$ variables.
Thus $\KK^\ve_\spot(M)$ is acyclic by exactness of colimits.
Moreover, again by algebraic independence, the nonzero homology of
$\KK^\ve_\spot(M)$ is naturally the $R^\yy$-module~$M^\yy$, with an
action of $\kk[\xx^{[\ve]}]$ where $x_i^\ve$ acts the same way
as~$y_i^\ve$ due to the relation $x_i^\ve - y_i^\ve$.
\end{proof}

\begin{cor}\label{c:fldim}
The $n$-variable real-exponent polynomial ring has flat dimension~$n$.
\end{cor}
\begin{proof}
Example~\ref{e:usual-koszul} implies that $\fldim R \geq n$, and
$\fldim R \leq n$ by Theorem~\ref{t:flat-res}.
\end{proof}

\section{Global dimension~\texorpdfstring{$n + 1$}{n+1}}\label{s:gldim}

\begin{lemma}\label{l:orthant-res}
Fix an orthant ideal~$\OO \neq R$.  Choose a sequence
$\{\ee_k\}_{k\in\NN}$ such that $\ee_k =
(\ve_{1k},\dots,\ve_{nk}) \in \RR^n_+$ has
\begin{itemize}
\item%
$\ve_{ik} = 0$ for all~$k$ if the $i^\mathrm{th}$ factor of~$\OO$
is~$R_1$ and
\item%
$\{\ve_{ik}\}_{k \in \NN}$ strictly decreases with limit~$0$ if the
$i^\mathrm{th}$ factor of~$\OO$ is~$\mm_1$.
\end{itemize}
Let $F = \bigoplus_k \<\xx^{\ee_k}\>$ be the direct sum of the
principal ideals in~$R$ generated by the monomials with
degrees~$\ee_k$.  Each summand $\<\xx^{\ee_k}\>$ is free with basis
vector~$1_k$, and $\OO$ has a free resolution \mbox{$0 \from F \from F
\from 0$} whose differential sends $1_k \in \<\xx^{\ee_k}\>$ to
\mbox{$1_k - \xx^{\ee_k - \ve_{k+1}}1_{k+1}$}.
\end{lemma}
\begin{proof}
The augmentation map $\OO \overset{\;\alpha}\ffrom F$ sends $1_k$
to~$\xx^{\ee_k}$.  It is surjective by definition of~$\OO$.  Since
$\alpha$ is graded by the monoid~$\RR^n_+$, its kernel can be
calculated degree by degree.  In degree $\aa \in \RR_+$ the kernel is
spanned by all differences $\xx^{\aa-\ee_k} 1_k -
\xx^{\aa-\ee_\ell} 1_\ell$ such that $\ee_k$ and~$\ee_\ell$
both weakly precede~$\aa$; indeed, this subspace of the $\aa$-graded
component~$F_\aa$ has codimension~$1$, and it is contained in the
kernel because $\xx^{\aa - \ee_k} \xx^{\ee_k} = \xx^{\aa -
\ee_\ell} \xx^{\ee_\ell}$.  The differential is injective
because each element $f \in F$ has nonzero coefficient on a basis
vector~$1_k$ with $k$ maximal, and the image of~$f$ has nonzero
coefficient on $1_{k+1}$.
\end{proof}

\begin{lemma}\label{l:Ext(k,F)}
$\kk = R_1/\mm_1$ has a free resolution of length~$2$, and
$\Ext^2_{R_1}(\kk,F) \neq 0$.
\end{lemma}
\begin{proof}
The resolution of~$\mm_1$ over~$R_1$ in Lemma~\ref{l:orthant-res}
(with $n = 1$) can be augmented and composed with the inclusion $R_1
\otni \mm_1$ to yield a free resolution of~$\kk$ over~$R_1$.  The long
exact sequence from $0 \from \kk \from R_1 \from \mm_1 \from 0$
implies that $\Ext^{i+1}_{R_1}(\kk,-) \cong \Ext^i_{R_1}(\mm_1,-)$ for
$i \geq 1$.  Now apply $\Hom(\mm_1,-)$ to the exact sequence $0 \to F
\to F \to \mm_1 \to 0$.  The first few terms are $0 \to \Hom(\mm_1,F)
\to \Hom(\mm_1,F) \to R_1 \to \Ext^1(\mm_1,F)$.  The image of
$\Hom(\mm_1,F) \to R_1$ is~$\mm_1$, so $\kk \into \Ext^1(\mm_1,F)
\cong \Ext^2(\kk,F)$ is nonzero.
\end{proof}

\begin{remark}\label{r:osofsky}
Any ideal that is a countable (but not finite) union of a chain of
principal ideals has projective dimension~$1$
\cite[p.14]{osofsky1974}.  But it is convenient to have an explicit
free resolution of~$\mm_1$ over~$R_1$, and it is no extra work to
resolve all orthant ideals.
\end{remark}

\begin{prop}\label{p:n+1}
Set $\mm_1 = \<x_n^\ve \mid \ve > 0\>$ and $J = \<x_1,\dots,x_{n-1}\>
\subseteq R$.  Using $x = x_n$ for~$R_1$, consider the $R_1$-module
$F$ in Lemma~\ref{l:Ext(k,F)} with $n = 1$ as an $R$-module via $R
\hspace{-.2ex}\onto\hspace{-.4ex} R_1$, where $x_i^\ve \mapsto 0$ for
all~$\ve > 0$ and $i \leq n-1$.  Then $\Ext^{n+1}_R(R/I,F) \neq 0$
when~$I = J + \mm_1$.
\end{prop}
\begin{proof}
Let $\FF_\spot: 0 \from R_1 \from F \from F \from 0$ be the $R_1$-free
resolution of~$\kk$ obtained by augmenting the resolution of~$\mm_1$
in Lemma~\ref{l:orthant-res} with $n = 1$.  Let $\KK_\spot =
\KK_\spot^{R_{n-1}}(\xx_{n-1})$ be the ordinary Koszul complex over
$R_{n-1}$ on the sequence $\xx_{n-1} = x_1,\dots,x_{n-1}$, which is a
free resolution of $R_{n-1}/\xx_{n-1} R_{n-1}$ over~$R_{n-1}$.  Then
$\Tot(\FF_\spot \otimes_\kk \KK_\spot)$ is a free resolution of $R/I$
over~$R$.  On the other hand,
\begin{align*}
  \FF_\spot \otimes_\kk \KK_\spot
 &\cong
  \FF_\spot\otimes_{R_1}R_1 \otimes_\kk R_{n-1}\otimes_{R_{n-1}}\KK_\spot
\\&\cong
  \FF_\spot\otimes_{R_1} R \otimes_{R_{n-1}}\KK_\spot
\\&\cong
  \FF_\spot\otimes_{R_1}R \otimes_R R\otimes_{R_{n-1}}\KK_\spot
\\&=
  \FF_\spot^R \otimes_R \KK_\spot^R,
\end{align*}
where $\FF_\spot^R = \FF_\spot\otimes_{R_1}R$ is an $R$-free
resolution of~$R/\mm_1 R$ and the ordinary Koszul complex $\KK_\spot^R
= R\otimes_{R_{n-1}}\KK_\spot = \KK_\spot^R(\xx_{n-1})$ of the
sequence $\xx_{n-1}$ in~$R$ is an $R$-free resolution of~$R/J$.

Using $(-)^*$ to denote the free dual $\Hom_R(-,R)$, compute
\begin{align}\label{eq:Hom}
\Hom_R(\FF_\spot^R \otimes_R \KK_\spot^R, F)
 &\cong
  \Hom_R\big(\FF_\spot^R, \Hom_R(\KK_\spot^R, F)\big)
\\\nonumber&\cong
  \Hom_R(\FF_\spot^R, (\KK_\spot^R)^* \otimes_R F)
\\\nonumber&\cong
  \Hom_R(\FF_\spot^R, (\KK_\spot^R)^* \otimes_R R_1 \otimes_{R_1} F),
\end{align}
where the bottom isomorphism is because the $R$-action on~$F$ factors
through~$R_1$.  The differentials of the complex $(\KK_\spot^R)^*
\otimes_R R_1 \cong (\KK_\spot^R)^* \otimes_{R_{n-1}} \kk$
all vanish, and this complex has cohomology $R_1^{\binom{n-1}q}$ in
degree~$q$.  Hence the total complex of~Eq.~(\ref{eq:Hom}) has
homology
\pagebreak[1]
\begin{align*}
\Ext^i_R(R/I,F)
 &\cong
  \bigoplus_{p+q=i} H_p \Hom_R\big(\FF_\spot^R, F^{\binom{n-1}q}\big)
\\&\cong
  \bigoplus_{p+q=i} H_p \Hom_{R_1}\big(\FF_\spot, F^{\binom{n-1}q}\big)
\\&\cong
  \bigoplus_{p+q=i} \Ext^p_{R_1}\big(\kk, F^{\binom{n-1}q}\big),
\end{align*}
where the middle isomorphism is again because the $R$-action on~$F$
factors through~$R_1$.  Taking $p = 2$ and $q = n-1$ yields the
nonvanishing by Lemma~\ref{l:Ext(k,F)}.
\end{proof}

\begin{remark}\label{r:grothendieck}
The proof of Proposition~\ref{p:n+1} is essentially a Grothendieck
spectral sequence for the derived functors of the composite
$\Hom_{R_1}(\kk,-) \circ \Hom_{R_{n-1}}(R_{n-1}/\xx_{n-1},-)$ but the
elementary Koszul argument isn't more lengthy
than verifying the hypotheses.
\end{remark}

\begin{thm}\label{t:gl.dim}
The $n$-variable real-exponent polynomial ring has global
dimension~\mbox{$n\!+\!1$}.\!
\end{thm}
\begin{proof}
Proposition~\ref{p:n+1} yields the lower bound $\gldim R \geq n + 1$.
For the opposite bound, given any $R$-module~$M$, each module in the
length~$n$ flat resolution from Theorem~\ref{t:flat-res} has a free
resolution of length at most~$1$ by Lemma~\ref{l:orthant-res}.  By the
comparison theorem for projective resolutions
\cite[Theorem~2.2.6]{weibel1994}, the differentials of the flat
resolution lift to chain maps of these free resolutions.  The total
complex of the resulting double complex has length at most~$n + 1$.
\end{proof}

\begin{remark}\label{r:pd(k)}
As an $\RR^n$-graded module, the quotient $R/I$ in
Proposition~\ref{p:n+1} is nonzero only in degrees from $\RR^{n-1}
\subseteq \RR^n$.  Hence $R/I$ is ephemeral \cite{berkouk-petit2021},
meaning more or less that its set of nonzero degrees has measure~$0$.
The projective dimension exceeding~$n$ is not due solely to this
ephemerality.  Indeed, multiplication by~$x_n^1$ induces an inclusion
of $R/I$ into $R/I'$ for $I' = \<x_1,\dots,x_{n-1}\> + \<x_n^\ve \mid
\ve > 1\>$, which is supported on a unit cube in~$\RR^n_+$ that is
neither open nor closed.
Theorem~\ref{t:gl.dim} implies that $\Ext^{n+1}_R(R/I',N)
\onto\nolinebreak \Ext^{n+1}_R(x_n R/I,N)$ is surjective for all
modules~$N$, so $R/I'$ has projective dimension $n+1$.  On the other
hand, it could be the closed right endpoints
\cite{functorial-endpoints}---that is, closed socle elements
\cite[Section~4.1]{essential-real}---that cause the problem.
Thus it could be
that sheaves in the conic topology (``\mbox{$\gamma$-topology}''; see
\cite{kashiwara-schapira2018,strat-conical,berkouk-petit2021}) have
consistently lower projective dimensions\mkl{.}
\end{remark}

\section{Dense exponent sets}\label{s:dense}

The results in Sections~\ref{s:fldim} and~\ref{s:gldim} extend to
monoid algebras for positive cones of subgroups of~$\RR^n$ satisfying
a mild density condition.  Applications to noncommutative toric
geometry should require restriction to subgroups of this sort.

\begin{defn}\label{d:G}
Let $G \subseteq \RR^n$ be a subgroup whose intersection with each
coordinate ray $\rho$ of\/~$\RR^n$ is dense.
Write \raisebox{0pt}[11pt][0pt]{$G_+ = G \cap \RR^n_+$} for the
positive cone in~$G$, set $\orho = \rho \cap \RR^n_+ \minus \{\0\}$,
and let \raisebox{0pt}[11pt][0pt]{$\oG_+ = \prod_\rho G \cap \orho_+$}
be the set of points in~$G$ whose projections to all coordinate rays
are strictly positive and still lie in~$G$.  Set $R_G = \kk[G_+]$, the
monoid algebra of~$G_+$ over~$\kk$.  Let $R_G^\xx$~and~$R_G^\yy$ be
the copies of~$R_G$ embedded in $R_G^{\otimes 2}$ as $R_G \otimes
1$~and~$1 \otimes R_G$.  For $\ee \in \oG_+$ let $\xx^{[\ee]} =
x_1^{\ve_1},\dots,x_n^{\ve_n}$ be the corresponding sequence of
elements in~$R_G$.
\begin{enumerate}
\item%
The \emph{open Koszul complex} over~$R_G$ is the colimit
$\oK_\spot^\xx = \dirlim_{\ee \in \oG_+} \KK_\spot(\xx^{[\ee]})$.
\vspace{.3ex}

\item%
Fix an $R_G^\xx$-module~$M$.  Write $M^\yy$ for the corresponding
$R_G^\yy$-module, with the $\xx$ variables renamed to~$\yy$.  With
notation for variables as in Lemma~\ref{l:Rotimes2}, the \emph{open
total Koszul complex} of~$M$ is the colimit $\oK^{\xx-\yy}_\spot(M) =
\dirlim_{\ee \in \oG_+} \KK_\spot(\xx^{[\ee]}-\nolinebreak\yy^{[\ee]})
\otimes_{R^\yy}\nolinebreak M^\yy$.

\item%
Given a subset $\sigma \subseteq \{1,\dots,n\}$, the \emph{orthant
ideal} $I_\sigma \subseteq R_G$ is generated by all monomials
$\xx^{\ee}$ for $\ee \in G_+$ such that $\ve_i > 0$ for all $i \in
\sigma$.
\end{enumerate}
\end{defn}

\begin{example}\label{e:dense}
Let $G$ be generated by
$\bigl[\twoline 20\bigr], \bigl[\twoline \pi 0\bigr], \bigl[\twoline
11\bigr], \bigl[\twoline 0e\bigr]$ as a subgroup of~$\RR^2$, so $G$
consists of the integer linear combinations of these four vectors.
The intersection $G \cap \rho^y$ with the $y$-axis $\rho^y$ arises
from integer coefficients $\alpha$, $\beta$, $\gamma$, and $\delta$
such that
$$%
  \bigl[\twoline 0y\bigr]
  =
  \alpha\bigl[\twoline 20\bigr] + \beta\bigl[\twoline \pi 0\bigr] +
  \gamma\bigl[\twoline 11\bigr] + \delta\bigl[\twoline 0e\bigr].
$$
This occurs precisely when $2\alpha + \pi\beta + \gamma = 0$, and in
that case $y = \gamma + \delta e$.  Since $\pi$ is irrational it is
linearly independent from~$1$ over the integers,
so $\beta = 0$ and hence $\gamma = -2\alpha$ is always an even
integer.  Since $e$ is irrational, the only integer points
in $G \cap \rho^y$ have even $y$-coordinate:
$$%
  G \cap \rho^y
  =
  \bigl\<\bigl[\twoline 02\bigr], \bigl[\twoline 0e\bigr]\bigr\>.
$$
The point $\bigl[\twoline 11\bigr] \in G$ has strictly positive
projection to~$\rho^y$, but that projection lands outside of~$G$.
Hence \raisebox{0pt}[11pt][0pt]{$\oG_+ = G \cap \orho^{\,x}_+ \times G
\cap \orho^{\,y}_+$} is a proper subgroup of~$G$, given the strictly
positive point \raisebox{0pt}[11pt][0pt]{$\bigl[\twoline 11\bigr] \in
G_+ \minus \oG_+$}.  Nonetheless, $\oG_+$
contains positive real multiples of~$\bigl[\twoline 11\bigr]$
approaching
the origin, which is all the colimit in the proof of
Theorem~\ref{t:G-flat-res}~requires.
\end{example}

\begin{thm}\label{t:G-flat-res}
If a subgroup $G \subseteq \RR^n$ is dense in every coordinate
subspace of\/~$\RR^n$ as in Definition~\ref{d:G}, then
Theorem~\ref{t:flat-res} holds verbatim with $R_G = \kk[G \cap
\RR^n_+]$ in place of~$R$.  Consequently, the ring~$R_G$ has flat
dimension~$n$ and global dimension~$n+1$.
\end{thm}
\begin{proof}
For $\sigma \subseteq \{1,\dots,n\}$ and $\ee \in \RR^n$ let
$\ee_\sigma \in \RR^n$ be the restriction of $\ee$ to~$\sigma$, so
$\ee_\sigma$ has entry~$0$ in the coordinate indexed by every $j
\not\in \sigma$.  The $2^n$ summands of $\oK_\spot^\xx$ are orthant
ideals because $\KK_i(\xx^{[\ee]}) \cong \bigoplus_{|\sigma|=i}
\<\xx^{\ee_\sigma}\>$ naturally with respect to the inclusions induced
by the colimit defining~$\oK_\spot^\xx$.  Each orthant ideal is flat
because this colimit is filtered: given two vectors $\ee_1, \ee_2 \in
\oG_+$, the coordinatewise minimum $\ee_1 \wedge \ee_2 \in \RR^n_+$
lies in~$\oG_+$ because its projection to each ray lies in~$G$.
Proposition~\ref{p:fldim(k)} therefore generalizes to~$R_G$ by the
exactness of colimits and the cokernel calculation $\kk = R_G/\mm$ for
the \mbox{$G$-graded} maximal ideal $\mm = \<\xx^\ee \mid \0 \neq \ee
\in G_+\>$.  Example~\ref{e:usual-koszul} generalizes with no
additional work.  Lemma~\ref{l:Rotimes2} generalizes by exactness of
colimits and the cokernel calculation $R_G \cong R_G^{\otimes
2}/\<\xx^{[\ee]} - \yy^{[\ee]} \mid \0 \neq \ee \in G_+\>$.  The
conclusion of Remark~\ref{r:orthant} generalizes, but the reason is
direct calculation of $\oK_\spot(\xx^{[\ee]}-\yy^{[\ee]})$ as was done
for~$\oK_i^\xx$.  The original proof of Theorem~\ref{t:flat-res} uses
that tensor products commute with colimits, but the generalized proof
avoids that argument by simply defining $\oK_\spot^{\xx-\yy}$ as the
relevant colimit.  The rest of the proof and the generalization of the
flat dimension claim in Corollary~\ref{c:fldim} work mutatis mutandis,
given the strengthened versions of the results~they~cite.

The orthant ideal resolution in Lemma~\ref{l:orthant-res} generalizes
to~$R_G$ by the density hypothesis, including specifically the part
about intersecting with coordinate subspaces.  The Ext calculation in
Lemma~\ref{l:Ext(k,F)} works again by density of the exponent set
of~$\mm_1$ in~$\RR_+$.  The statement and proof of
Proposition~\ref{p:n+1} work mutatis mutandis for $R_G$ in place
of~$R$ as long as the power of~$x_i$ generating~$J$ lies in the
intersection of~$G$ with the corresponding coordinate ray of~$\RR^n$.
The proof of Theorem~\ref{t:gl.dim} then works verbatim, given the
strengthened versions of the results it cites.
\end{proof}



\end{document}